\documentclass{lms}
\usepackage{latexsym}
\usepackage{amssymb,latexsym,amscd,amsmath,mathtools}
\usepackage{mathrsfs}
\usepackage{setspace}
\usepackage{url}

\newtheorem{theorem}{Theorem}[section] 
\newtheorem{lemma}[theorem]{Lemma}     
\newtheorem{corollary}[theorem]{Corollary}
\newtheorem{proposition}[theorem]{Proposition}

\newcommand{\be}{\begin{equation}}
\newcommand{\ee}{\end{equation}}

\def\N{\mathbb{N}}

\def\R{\mathbb{R}}

\def\Z{\mathbb{Z}}
\def\E{\mathbb{E}}
\def\cI{\mathcal{I}}
\def\cL{\mathcal{L}}
\def\cT{\mathcal{T}}

\providecommand{\re}{\mathop{\rm Re}}
\providecommand{\im}{\mathop{\rm Im}}

\title[An upper bound for moments of $\zeta'(\rho)$]{An upper bound for discrete moments of the derivative of the Riemann zeta-function}

\author{S. Kirila}

\classno{11M06, 11M26}

\begin{document}
	\maketitle
\begin{abstract}
	Assuming the Riemann hypothesis, we establish an upper bound for the $2k$-th discrete moment of the derivative of the Riemann zeta-function at nontrivial zeros, where $k$ is a positive real number. Our upper bound agrees with conjectures of Gonek and Hejhal and of Hughes, Keating, and O'Connell. This sharpens a result of Milinovich. Our proof builds upon a method of Adam Harper concerning continuous moments of the zeta-function on the critical line.
\end{abstract}

\section{Introduction}

The estimation of various types of moments of the Riemann zeta-function has been intensely studied for the better part of a century. The zeta-function is given by
\[
\zeta(s)=\sum_{n=1}^{\infty}n^{-s},
\]
where $s=\sigma+it$ denotes a complex variable with real part $\sigma$ and imaginary part $t$. This definition is valid for $\sigma>1$, but $\zeta(s)$ may be continued analytically to the rest of the complex plane except for a simple pole at $s=1$. The specific moments we examine here are defined as
\[
J_k(T):=\frac{1}{N(T)}\sum_{0<\im(\rho)\le T}|\zeta'(\rho)|^{2k}
\]
for $k$ a positive real number. The sum here is over nontrivial zeros $\rho$ of the zeta-function, i.e. those zeros with positive real part. The normalizing factor $N(T)$ is the number of $\rho$ over which we are summing, so $J_k(T)$ is the $2k$-th moment of $|\zeta'(s)|$ on the discrete probability space
\[
\{\rho:\zeta(\rho)=0,\re(\rho)>0\text{ and }0<\im(\rho)\le T\}
\]
equipped with the uniform measure. Consequently $J_k(T)$ is commonly called a discrete moment in the literature. The more information we have regarding $J_k(T)$, the more we can say about the distribution of values of $|\zeta'(\rho)|$.

These discrete moments were first studied by Gonek~\cite{Go84}, who conditionally established the asymptotic formula
\[
J_1(T)\sim \tfrac{1}{12}(\log T)^3
\]
with an explicit error term. Gonek's proof relies inherently on the validity of the Riemann hypothesis (RH), the statement that all nontrivial zeros $\rho$ have real part $\re(\rho)=\frac12$. No asymptotic formulas are known for other values of $k$, even assuming RH, and Gonek's estimate has yet to be proved unconditionally. Gonek~\cite{Go89} and Hejhal~\cite{He} independently conjectured
\be\label{GoHe conjecture}
J_k(T)\asymp(\log T)^{k(k+2)}
\ee
for any real number $k$. This agrees with Gonek's estimate for $J_1(T)$. Ng~\cite{Ng} proved $J_2(T)\asymp (\log T)^8$, so the conjecture also holds for $k=2$.
This conjecture has since been strengthened. Using a random matrix model for $\zeta'(\rho)$, Hughes, Keating, and O'Connell~\cite{HuKeOc} suggested the asymptotic formula
\[
J_k(T)\sim C_k(\log T)^{k(k+2)}.
\]
The constants $C_k$ in their conjecture are explicit, given by
\[
C_k=\frac{G(k+2)^2}{G(2k+3)}\prod_{p\text{ prime}}\left(1-\frac{1}{p}\right)^{k^2}\sum_{m=0}^{\infty}\left(\frac{\Gamma(m+k)}{m!\,\Gamma(k)}\right)^2p^{-m},
\]
where $G(x)$ is the Barnes $G$-function. Furthermore, they provided a heuristic explanation which suggests that the conjectured asymptotic formula of Gonek and Hejhal should fail for $k\le-\frac32$. Hughes, Keating, and O'Connell inserted the product over primes here was inserted in an \emph{ad hoc} manner, namely from a heuristic estimate for the case $k=-1/2$. Bui, Gonek, and Milinovich~\cite{BuGoMi} used a hybrid Euler-Hadamard product model for $\zeta'(\rho)$ to suggest precisely where this product over primes comes from, essentially merging ideas from number theory and random matrix theory in the same way as Gonek, Hughes, and Keating~\cite{GoHuKe} did for moments of $\zeta(\frac12+it)$.

These conjectures remain open, but work has been done toward the implied upper and lower bounds conditionally on RH. Milinovich and Ng~\cite{MiNg} obtained the expected lower bound
\be\label{MiNg}
J_k(T)\gg_k (\log T)^{k(k+2)}
\ee
for any natural number $k$. In the other direction, Milinovich\cite{Mi} showed that
\be\label{Mi}
J_k(T)\ll_{k,\varepsilon} (\log T)^{k(k+2)+\varepsilon}
\ee
for any $\varepsilon>0$. The purpose of this paper is to remove the $\varepsilon$ in the exponent here and prove the following.

\begin{theorem}\label{Thm 1}
	Assume RH. Let $k>0$. Then $$J_k(T)\ll_k(\log T)^{k(k+2)}$$
	as $T\to \infty$.
\end{theorem}
\noindent Together with \eqref{MiNg}, this shows that
\[
J_k(T)\asymp_k (\log T)^{k(k+2)}
\]
for $k$ a natural number. This proves (on RH) the conjecture of Gonek and Hejhal for $k$ a positive natural number. In fact, the method of Milinovich and Ng used to prove \eqref{MiNg} should be able to cover the case of real $k>0$ using the work of Radziwi\l\l\ and Soundararajan in \cite{RaSo} and \cite{RaSo2}, assuming RH. This would establish the conjecture of Gonek and Hejhal for all real, positive $k$. We remark that the implied constant in Theorem~\ref{Thm 1}  grows like $ e^{e^{Ak}}$ for some $A>0$ as $k$ gets large. For comparison, the conjecture of Hughes, Keating, and O'Connell suggests an implied constant $\approx e^{-k^2\log k}$ is permissible.

In the last section of the paper, we shall also indicate how to prove the following result.

\begin{theorem}\label{Thm 2}
	Assume RH. Let $k>0$. Let $\alpha$ be a complex number with $|\alpha|\le (\log T)^{-1}$. Then
	\[
	\frac{1}{N(T)}\sum_{0<\im(\rho)\le T}|\zeta(\rho+\alpha)|^{2k}\ll_k (\log T)^{k^2}
	\]
	as $T\to \infty$.
\end{theorem}
\noindent This shifted moment of the zeta-function was considered by Milinovich~\cite{Mi} in his proof of \eqref{Mi}, and our Theorem~\ref{Thm 2} is an improvement of Theorem 1.2 in \cite{Mi}. As a consequence of our result, we deduce the following.

\begin{corollary}
	Assume RH. Let $k\ge \frac12$ and let $\nu$ be a positive integer. Then
	\[
	\frac{1}{N(T)}\sum_{0<\im(\rho)\le T}|\zeta^{(\nu)}(\rho)|^{2k}\ll_{k,\nu} (\log T)^{k(k+2\nu)}
	\]
	as $T\to \infty$.
\end{corollary}
\begin{proof}
	This follows from the direct analogue of Lemma 8.1 in \cite{MiNg2}; we provide the details here for the sake of completeness. By Cauchy's integral formula, we have
	\be\label{eq:cauchy}
	\sum_{0<\im(\rho)\le T}|\zeta^{(\nu)}(\rho)|^{2k}=\Big(\frac{\nu!}{2\pi}\Big)^{2k}\sum_{0<\im(\rho)\le T}\bigg|\int_C\frac{\zeta(\rho+\alpha)}{\alpha^{\nu+1}}d\alpha\bigg|^{2k},
	\ee
	where $C$ is the positively-oriented circle of radius $(\log T)^{-1}$ centered at the origin. Since
	\[
	\bigg|\int_C\frac{\zeta(\rho+\alpha)}{\alpha^{\nu+1}}d\alpha\bigg|\le (\log T)^{\nu+1}\int_C|\zeta(\rho+\alpha)|\,|d\alpha|,
	\]
	it follows from \eqref{eq:cauchy} that
	\[
	\sum_{0<\im(\rho)\le T}|\zeta^{(\nu)}(\rho)|^{2k}\le \Big(\frac{\nu!}{2\pi}\Big)^{2k}(\log T)^{2k(\nu+1)}\sum_{0<\im(\rho)\le T}\bigg(\int_C|\zeta(\rho+\alpha)|\,|d\alpha|\bigg)^{2k}.
	\]
	If $k>\frac12$, then H\"older's inequality implies that
	\[
	\bigg(\int_C|\zeta(\rho+\alpha)|\,|d\alpha|\bigg)^{2k}\le \bigg(\int_C|d\alpha|\bigg)^{2k-1}\bigg(\int_C|\zeta(\rho+\alpha)|^{2k}|d\alpha|\bigg);
	\]
	when $k=\frac12$, this same bound trivially holds. The first integral on the right-hand side here is precisely $2\pi (\log T)^{-1}$, and so we see that
	\[
	\sum_{0<\im(\rho)\le T}|\zeta^{(\nu)}(\rho)|^{2k}\le \frac{(\nu !)^{2k}}{2\pi}(\log T)^{2k\nu+1}\int_C\bigg(\sum_{0<\im(\rho)\le T}|\zeta(\rho+\alpha)|^{2k}\bigg)|d\alpha|.
	\]
	The result follows by dividing both sides of this last inequality by $N(T)$, applying Theorem~\ref{Thm 2} to the sum on the right-hand side, and finally integrating over $\alpha$.
\end{proof}
Lastly, we remark on some connections between these discrete moments and simple zeros of $\zeta(s)$. If $N^*(T)$ counts the number of simple zeros with $0<\im(\rho)\le T$, then Cauchy--Schwarz implies that
\[
N^*(T)\ge \frac{J_k(T)^2}{J_{2k}(T)}.
\]
Montgomery's~\cite{Mo} pair correlation conjecture implies that almost all zeros are simple, and it is generally expected that this is true of all zeros. However, $J_k(T)$ grows too quickly as $T\to \infty$ to obtain even a positive proportion of simple zeros via the inequality above. In order to minimize the loss from Cauchy-Schwarz, Conrey, Ghosh and Gonek~\cite{CoGhGo} used a mollified version of $J_1(T)$ to show that at least $19/27$ of the nontrivial zeros are simple, assuming the generalized Riemann Hypothesis\footnote{The statement of their result actually assumes RH and the generalized Lindel\"of Hypothesis. However, it appears that there is a problem with their proof under these assumptions. This may be resolved by assuming the generalized RH instead. We thank Professors Gonek and Milinovich for bringing this to our attention.} for Dirichlet $L$-functions; Bui and Heath-Brown subsequently proved the same result assuming only RH. It may be of future interest to estimate mollified versions of $J_k(T)$ for $k>1$, though it seems unlikely that this will lead to significant improvements on the proportion of simple zeros.

Another connection between $J_k(T)$ and simple zeros is as follows. The Mertens function $M(x)$ is defined as
\[
M(x):=\sum_{n\le x}\mu(n),
\]
where $\mu$ is the M\"obius function. It is well known that RH is equivalent to the estimate $M(x)\ll x^{1/2+\varepsilon}$ for $\varepsilon>0$. In unpublished work, Gonek proved that RH and the conjectured upper bound $J_{-1}(T)\ll T$ from \eqref{GoHe conjecture} above imply that $M(x)\ll x^{1/2}(\log x)^{3/2}$, which was later shown by Ng~\cite{Ng} as well. Note that the bound $J_{-1}(T)\ll T$ automatically assumes that all zeros are simple, as otherwise $J_{-1}(T)=\infty$ for all sufficiently large $T$. In fact, under these hypotheses, Ng shows that $e^{-y/2}M(e^y)$ has a limiting distribution, with $0\le y\le Y$, as $Y\to \infty$. This, along with some additional assumptions which include an upper bound on $J_{-1/2}(T)$, leads Ng to re-establish the unpublished conjecture of Gonek that
\[
\underline{\overline{\lim}}_{x\to \infty}\frac{M(x)}{\sqrt{x}(\log\log\log x)^{5/4}}=\pm B
\]
for some positive constant $B$. Thus the study of $J_k(T)$ for $k$ small and negative may lead to further insight into the distribution and behavior of $M(x)$. 


\section{The idea behind the proof of Theorem~\ref{Thm 1}}

Let
\[
I_{k}(T)=\frac{1}{T}\int_0^T |\zeta(1/2+it)|^{2k} dt.
\]
In 2009  Soundararajan~\cite{So} showed that, on RH,
$I_{k}(T)\ll T(\log T)^{k^2+\epsilon}$ for every $\epsilon >0$. A few years later, Harper~\cite{Ha} devised a method  to prove, again on RH,  that  $I_{k}(T)\ll T(\log T)^{k^2}$, which is the actual conjectured size; moreover, it is the same size (in the $T$ aspect) of the unconditional lower bound proved by Radziwi{\l\l} and Soundararajan~\cite{RaSo}. Our proof of Theorem~\ref{Thm 1}   is  based on  Harper's method, and improves upon Milinovich's upper bound \eqref{Mi} in the same way that Harper's improves upon  Soundararajan's. We note that our implied constant is of the same form as that of Harper.

Harper's method relies on two ingredients. The first  is an upper bound for $\log|\zeta(\frac12+it)|$ in terms of a Dirichlet polynomial. The second is an estimate for integrals of the form
\be\label{harperintegral1}
\int_T^{2T}\cos(t\log p_1)\cdots \cos(t\log p_m)dt
\ee
for (not necessarily distinct) prime numbers $p_1,\ldots, p_m$. This follows easily from  the basic orthogonality estimate
\be\label{harperintegral2}
\int_T^{2T}e^{irt}dt=T\delta_0(r)+O(r^{-1}),
\ee
where $\delta_0$ is a Dirac mass at $0$.

Henceforth we write $\gamma$ in place of $\im(\rho)$, so that $\rho=\frac12+i\gamma$ assuming RH.
To estimate our discrete moments $J_k(T)$, our first ingredient is an upper bound for $\log |\zeta'(\rho)|$; our second, analogous to \eqref{harperintegral1}, is an estimate for sums of the form
\[
\sum_{0<\gamma\le T}\cos(\gamma\log p_1)\cdots \cos(\gamma\log p_m).
\]
The discrete analogue of \eqref{harperintegral2} is given by Gonek's~\cite{Go93} uniform version of Landau's formula. On RH, this says roughly that
\[
\sum_{0<\gamma\le T}e^{ir\gamma}=N(T)\delta_0(r)-Tf(r)+\text{ small error},
\]
where $f$ is a certain nonnegative function. Note that, in comparison with \eqref{harperintegral2}, this has a secondary term. The final contribution to $J_k(T)$ from this secondary term  is possibly of the same order as that of the first term, namely $N(T)(\log T)^{k(k+2)}$. However,   the  secondary term  contribution is not positive, so we may ignore it and still obtain an upper bound. 


The expectation that $J_k(T)\approx N(T)(\log T)^{k(k+2)}$ can be explained in a few different ways; here we discuss two heuristics. On average (in the sense of mean-square), $|\zeta'(\frac12+it)|$ is roughly $\log T$ larger than $|\zeta(\frac12+it)|$ in the interval $[0,T]$. Ford and Zaharescu~\cite{FoZa} showed that the $\gamma$ are equidistributed (mod $1$), so it seems reasonable to expect that the sum over $\gamma$ is approximated well by the corresponding integral, i.e.
\begin{align*}
	J_k(T) &\approx \frac{1}{T}\int_0^T|\zeta'(\tfrac12+it)|^{2k}dt\\
	&\approx (\log T)^{2k}I_k(T).
\end{align*}
A second heuristic relies on the expected Gaussian behavior of $\log|\zeta'(\rho)|$. Assuming RH and that the zeros of $\zeta(s)$ do not cluster together often in a particular sense, Hejhal~\cite{He} proved a central limit theorem for $\log|\zeta'(\rho)|$. Roughly speaking, he showed that the values of
$\log |\zeta'(\rho)|$, $0<\gamma\le T$, tend to be distributed like those of a Gaussian with mean $\log\log T$ and variance $\frac12 \log\log T$ as $T$ gets large, i.e. (see Theorem 4 of \cite{He})
\[
\lim_{T\to \infty}\frac{1}{N(T)}\#\left\{0<\gamma\le T:\frac{\log|\zeta'(\rho)|-\log\log T}{\sqrt{\frac12 \log\log T}}\in(a,b)\right\}=\frac{1}{\sqrt{2\pi}}\int_a^be^{-x^2/2}dx.
\]
If we assume that this central limit behavior holds uniformly in $T$, then this suggests that
\begin{align*}
J_k(T)&=\frac{1}{N(T)}\sum_{0<\gamma\le T}e^{2k\log |\zeta'(\rho)|}\\
&\approx \frac{1}{\sqrt{\pi\log\log T}}\int_{-\infty}^{\infty}e^{2kv}e^{-\frac{(v-\log\log T)^2}{\log\log T}}dv.
\end{align*}
After centering the integrand via the substitution $v\mapsto v+\log\log T$, this becomes
\[
\frac{(\log T)^{2k}}{\sqrt{\pi\log\log T}}\int_{-\infty}^{\infty}e^{2kv-\frac{v^2}{\log\log T}}dv.
\]
Completing the square then leads us to conclude that
\[
J_k(T)\approx (\log T)^{2k}(\log T)^{k^2},
\]
which matches the expected size. This should be compared with the estimates for $I_k(T)$ mentioned above, as Selberg's~\cite{Se,Se1} central limit theorem says that $\log |\zeta(\frac12+it)|$ tends to be distributed like a Gaussian with mean $0$ and variance $\frac12\log\log T$. Both heuristics suggest that the factor of $(\log T)^{2k}$ for discrete moments can be attributed to this difference in mean when compared with $I_k(T)$. 

\section{An upper bound for $\log|\zeta'(\rho)|$}
\noindent Throughout the paper, we denote prime numbers with the letters $p$ or minor variations such as $\tilde{p}$. When we write $p^h$, it is to be understood that $h$ is a natural number. The von Mangoldt function $\Lambda(n)$ is defined as
\[
\Lambda(n)=
\begin{cases}
\log p &\hbox{if }n=p^h,\\
0 &\hbox{otherwise}.
\end{cases}
\]
We extend the von Mangoldt function to the rest of $\R$ by taking $\Lambda(x)=0$ if $x$ is not a natural number; this will be useful in Lemma~\ref{Lem-Landau} below.
We also define a slight variant of $\Lambda(n)$. 
Let $\cL =\log T$. We
set
\[
\Lambda_{\cL}(n)=
\begin{cases}
\Lambda(n)&\hbox{if }n=p\; \text{or}\;  p^2 \;\text{and} \; n\le \cL,\\
0&\hbox{otherwise}.
\end{cases}
\]
We now prove the following upper bound for $\log|\zeta'(\rho)|$. This result is similar to upper bounds for $\log|\zeta(\frac12+it)|$ due to Soundararajan~\cite{So} and Harper~\cite{Ha}, and our proof is a modification of their arguments.

\begin{proposition}\label{Prop 1}
	Assume RH. Let $T$ be large and let $2\le x\le T^2$. Set $\sigma_x=\frac12+\frac{1}{\log x}$. If $\rho=\frac12+i\gamma$ is a zero of $\zeta(s)$ with $T< \gamma\le 2T$, then
	\[
	\log|\zeta'(\rho)|\le \re \sum_{n\le x}\frac{\Lambda_{\cL}(n)}{n^{\sigma_x+i\gamma}\log n}\frac{\log (x/n)}{\log x}+\log\log T+\frac{\log T}{\log x}+O(1).
	\]
\end{proposition}

\begin{proof}
	The inequality is true if $\zeta'(\rho)=0$, since the left-hand side is $-\infty$ and the right-hand side is finite. Thus we may assume that $\rho$ is a simple zero. We begin with the estimate
	\be\label{logderivative}
	-\re\frac{\zeta'}{\zeta}(\sigma+it)=\tfrac12\log T-\sum_{\tilde{\rho}=\frac12+i\gamma}\frac{\sigma-1/2}{(\sigma-1/2)^2+(t-\tilde{\gamma})^2}+O(1).
	\ee
	This follows from the Hadamard product formula for $\zeta(s)$ along with Stirling's approximation (see (4) in \cite{So}), and it is valid for $T\le t\le 2T$ as long as $t$ is not the ordinate of a zero of the zeta-function. Integrating $\sigma$ from $\frac12$ to $\sigma_x$ in \eqref{logderivative}, we have	
	\begin{multline}\label{prop1}
	\log|\zeta(\tfrac12+it)|-\log|\zeta(\sigma_x+it)|\\
	=(\sigma_x-\tfrac12)\big(\tfrac12\log T+O(1)\big)
	-\tfrac12\sum_{\tilde{\rho}=\frac12+i\tilde{\gamma}}\log \frac{(\sigma_x-\tfrac12)^2+(t-\tilde{\gamma})^2}{(t-\tilde{\gamma})^2}.
	\end{multline}
	Isolating the term corresponding to $\rho$ from the sum over zeros and subtracting $\log |t-\gamma|$ from both sides of \eqref{prop1}, we find that
	\begin{multline*}
	\log\bigg|\frac{\zeta(\tfrac12+it)}{t-\gamma}\bigg|-\log|\zeta(\sigma_x+it)|\\
	=(\sigma_x-\tfrac12)\big(\tfrac12\log T+O(1)\big)-\log|(\sigma_x-\tfrac12)+i(t-\gamma)|-\tfrac12\sum_{\tilde{\rho}\neq \rho}\log\frac{(\sigma_x-\tfrac12)^2+(t-\tilde{\gamma})^2}{(t-\tilde{\gamma})^2}.
	\end{multline*}
	Since $\rho$ is a simple zero, we may take the limit as  $t\to \gamma$ to obtain
	\begin{multline}\label{prop2}
	\log|\zeta'(\rho)|-\log|\zeta(\sigma_x+i\gamma)|\\
	=(\sigma_x-\tfrac12)\big(\tfrac12\log T+O(1)\big)
	-\log\big|\sigma_x-\tfrac12\big|-\tfrac12\sum_{\tilde{\rho}\neq \rho}\log \frac{(\sigma_x-\tfrac12)^2+(\gamma-\tilde{\gamma})^2}{(\gamma-\tilde{\gamma})^2}.
	\end{multline}
	Now define
	\be\label{F}
	\tilde{F}_x(\rho)=\sum_{\tilde{\rho}\neq \rho}\frac{\sigma_x-\tfrac12}{(\sigma_x-\tfrac12)^2+(\gamma-\tilde{\gamma})^2}.
	\ee
	Observe that this  sum is positive as $\sigma_x=\tfrac12+\frac{1}{\log x}$. Since $\log(1+x)\geq x/(1+x)$ for $x>0$, it follows from \eqref{prop2} that
	\be\label{prop3}
	\log|\zeta'(\rho)|\le \log |\zeta(\sigma_x+i\gamma)|+\log\log x-\tfrac12(\sigma_x-\tfrac12)\tilde{F}(\rho)+\frac12\frac{\log T}{\log x}+O(1).
	\ee
	Now we recall Lemma 1 of \cite{So}, which says that
	\[
	\frac{\zeta'}{\zeta}(s)\log x=-\sum_{n\le x}\frac{\Lambda(n)}{n^s}\log(x/n)-\bigg(\frac{\zeta'}{\zeta}(s)\bigg)'+\frac{x^{1-s}}{(1-s)^2}-\sum_{\tilde{\rho}}\frac{x^{\tilde{\rho}-s}}{(\tilde{\rho}-s)^2}-\sum_{k=1}^{\infty}\frac{x^{-2k-s}}{(2k+s)^2}
	\]
	for $x\ge 2$ and any $s$ not coinciding with $1$ or a zero of $\zeta(s)$. 
	The third term on the right-hand side and the last sum here are both $\ll x^{1-\sigma}/T^2$. Thus, after dividing by $\log x$, integrating $\sigma$ from $\infty$ to $\sigma_x$ and taking real parts of the resulting expressions, it follows that
	\be\label{log}
	\log |\zeta(s_x)|=\re\sum_{n\le x}\frac{\Lambda(n)}{n^s\log n}\frac{\log (x/n)}{\log x}-\frac{1}{\log x}\re \frac{\zeta'}{\zeta}(s_x)+\frac{1}{\log x}\re\sum_{\tilde{\rho}}\int_{\sigma_x}^{\infty}\frac{x^{\tilde{\rho}-s}}{(\tilde{\rho}-s)^2}d\sigma+O(1),
	\ee
	where $s_x=\sigma_x+i\gamma$. Recalling the definition of $\tilde{F}_x(\rho)$ from \eqref{F} above, we estimate the sum over $\tilde{\rho}\neq \rho$ in \eqref{log} as
	\begin{align*}
	\bigg|\re\sum_{\tilde{\rho}\neq \rho}\int_{\sigma_x}^{\infty}\frac{x^{\tilde{\rho}-s}}{(\tilde{\rho}-s)^2}d\sigma\bigg| &\le \frac{1}{\log x}\sum_{\tilde{\rho}\neq \rho}\frac{x^{\frac12-\sigma_x}}{|(\sigma_x-\frac12)+i(\gamma-\tilde{\gamma})|^2}\\
	&=\frac{x^{\frac12-\sigma_x}}{(\sigma_x-\frac12)\log x}\tilde{F}_x(\rho).
	\end{align*}
	Also, we may use \eqref{logderivative} to see that
	\[
	-\re\frac{\zeta'}{\zeta}(s_x)=\tfrac12\log T-\tilde{F}_x(\rho)-\frac{1}{\sigma_x-\frac12}.
	\]
	Applying both of these estimates to the right-hand side of \eqref{log}, we obtain	
	\begin{multline}\label{prop4}
	\log|\zeta(\sigma_x+i\gamma)|\le \re \sum_{n\le x}\frac{\Lambda(n)}{n^{\sigma_x+i\gamma}\log n}\frac{\log(x/n)}{\log x}-\frac{\tilde{F}_x(\rho)}{\log x}+\frac{1}{\log x}\int_{\sigma_x}^{\infty}\frac{x^{\frac12-\sigma}}{(\sigma-\frac12)^2}d\sigma\\
	-\frac{1}{(\sigma_x-\tfrac12)\log x}+\frac{x^{\frac12-\sigma_x}\tilde{F}(\rho)}{(\sigma_x-\frac12)\log^2 x}+\frac{\log T}{\log x}+O(1).
	\end{multline}
	After a change of variables, the integral in \eqref{prop4} may be expressed as
	\[
	\log x\int_1^{\infty}\frac{e^{-u}}{u^2}du.
	\]
	Hence the last term on the first line in \eqref{prop4} is a constant. Using \eqref{prop4} to estimate $\log|\zeta(s_x)|$ in \eqref{prop3} and recalling that $\sigma_x=\frac12+\frac{1}{\log x}$ and $x\le T^2$, we obtain
	\[
	\log|\zeta'(\rho)|\le \re \sum_{n\le x}\frac{\Lambda(n)}{n^{\sigma_x+i\gamma}\log n}\frac{\log(x/n)}{\log x}+\frac{\tilde{F}(\rho)}{\log x}(e^{-1}-\tfrac32)+\log\log T + \frac{\log T}{\log x}+O(1).
	\]
	Since $\tilde{F}(\rho)>0$ and $e^{-1}-\frac32<0$, we may omit the second term on the right-hand side here and still have an upper bound for the left-hand side. That is, we have
	\be\label{prop5}
	\log|\zeta'(\rho)|\le \re\sum_{n\le x}\frac{\Lambda(n)}{n^{\sigma_x+i\gamma}\log n}\frac{\log (x/n)}{\log x}+\log\log T+\frac{\log T}{\log x}+O(1).
	\ee
	The sum in \eqref{prop5} is supported on prime powers, and the prime powers $n=p^m$ with $m\geq 3$ contribute $O(1)$. Furthermore, as noted by Harper~\cite{Ha}, the sum over $n=p^2$ for $\log T<p\le \sqrt{x}$ is also bounded. Consequently, we conclude that
	\[
	\log|\zeta'(\rho)|\le \re\sum_{n\le x}\frac{\Lambda_{\cL}(n)}{n^{\sigma_x+i\gamma}\log n}\frac{\log (x/n)}{\log x}+\log\log T+\frac{\log T}{\log x}+O(1).
	\]
	This completes the proof of Proposition~\ref{Prop 1}.
\end{proof}

\section{Notation and Setup}
\noindent Let $N(T,2T)$ denote the number of $\rho=\frac12+i\gamma$ with $T<\gamma\le 2T$, i.e. $$N(T,2T)=N(2T)-N(T).$$
Our approach is to prove the following upper bound for discrete moments on dyadic intervals.
\begin{proposition}\label{Prop 2}
	Assume RH. Let $k>0$. Then
	\[\frac{1}{N(T,2T)}\sum_{T<\gamma\le 2T}|\zeta'(\rho)|^{2k}\ll_k(\log T)^{k(k+2)}\]
	as $T\to \infty$.
\end{proposition}
\noindent Theorem~\ref{Thm 1} follows from Proposition~\ref{Prop 2}. To see this, first divide the interval $(0,T]$ into dyadic subintervals $(2^{-i}T,2^{1-i}T]$ for $i\geq 1$. Second, note that
\[N(2^{-i}T,2^{1-i}T]\asymp 2^{1-i}N(T);\]
this follows from the Riemann-von Mangoldt formula (see Ch. 15 of \cite{Da}), which says
\[N(T)=\tfrac{T}{2\pi}\log \tfrac{T}{2\pi e}+O(\log T).\]
Applying Proposition~\ref{Prop 2} to each subinterval and summing over $i$ yields the conclusion of Theorem~\ref{Thm 1}.

In order to prove Proposition~\ref{Prop 2}, we begin by defining an increasing geometric sequence $\{\beta_i\}$ of real numbers by
\[
\beta_i=
\begin{cases}
\frac{20^{i-1}}{(\log\log T)^2}&\hbox{if }i\geq 1,\\
\hfil 0 &\hbox{if }i=0.
\end{cases}
\]
We will not need all $i\geq 0$, and we take the upper threshold of the index as
\[
\cI:=\max\{i:\beta_i\le e^{-1000k}\}.
\]
We split $(0,T^{\beta_{\cI}}]$ into disjoint subintervals $I_i=(T^{\beta_{i-1}},T^{\beta_i}]$ for $1\le i\le \cI$ and define $$w_j(n)=\frac{\Lambda_{\cL}(n)}{n^{1/\beta_j\log T}\log n}\frac{\log (T^{\beta_j}/n)}{\log T^{\beta_j}}$$
for $1\le j\le \cI$. Setting
\[
G_{i,j}(t)=\re\sum_{n\in I_i}\frac{w_j(n)}{\sqrt{n}}n^{-it}
\]
for $1\le i\le j\le \cI$, the conclusion of Proposition~\ref{Prop 1} can be written
\be\label{newinequality}
\log |\zeta'(\rho)|\le \re\sum_{i=1}^jG_{i,j}(\gamma)+\log\log T+\beta_j^{-1}+O(1).
\ee
We also need a particular random model for $G_{i,j}(\gamma)$. Let $\{X_p\}$ be a sequence of independent random variables indexed by the primes, where each $X_p$ is uniformly distributed on the unit circle in the complex plane. If $n$ has prime factorization $n=p_1^{h_1}\cdots p_r^{h_r}$, then we define
\[
X_n=X_{p_1}^{h_1}\cdots X_{p_r}^{h_r}.
\]
Thus $X_n$ is a random completely multiplicative function. We then define the random model $G_{i,j}(X)$ as
\[
G_{i,j}(X)=\re \sum_{n\in I_i}\frac{w_j(n)}{\sqrt{n}}X_n.
\]

Next we sort the $\gamma$ in the interval $[T,2T]$ into subsets based on the size of $G_{i,j}(\gamma)$. 
First let
\be\label{defT}
\cT=\{T<\gamma\le 2T:|G_{i,\cI}(\gamma)|\le \beta_i^{-3/2}\text{ for } 1\le i\le \cI\}.
\ee
This can be thought of as the \emph{best} set of $\gamma$, those for which $\exp 2k\re G_{i,\cI}(\gamma)$ can be approximated well by a short truncation of its Maclaurin series for every $1\le i\le \cI$ (see Lemma~\ref{Lem-Taylor} below). Similarly we define 
\begin{multline}\label{defS(j)}
S(j)=\{T<\gamma\le T:|G_{i,\ell}(\gamma)|\le \beta_i^{-3/4}\text{ for }1\le i\le j\text{ and }i\le \ell\le \cI,\\
\text{ but }|G_{j+1,\ell}(\gamma)|>\beta_{j+1}^{-3/4}\text{ for some }j+1\le \ell\le \cI\}
\end{multline}
for $1\le j\le \cI-1$. The remaining subset $S(0)$ is
\be\label{defS(0)}
S(0)=\{T<\gamma\le 2T:|G_{1,\ell}(\gamma)|>\beta_1^{-3/4}\text{ for some }1\le \ell \le \cI\}.
\ee
In a certain sense, the sets $S(j)$ ($1\le j<\cI$) are not as \emph{good} as $\cT$, but they are not as \emph{bad} as $S(0)$. This is evident in the fact that Lemma~\ref{Lem-Taylor} below does not say anything about $S(0)$. However, we will see in \S 6.3 that the contribution of $\gamma\in S(0)$ in Proposition~\ref{Prop 2} is negligible.

\section{Some lemmas}

\noindent The main ingredient in our proof is a uniform version of Landau's formula~\cite{La}. This was originally proved by Gonek\cite{Go93} and was studied in further detail by many others (e.g. \cite{FoSoZa}\cite{FoZa}\cite{Fu}). The version we use here is essentially the one found in \cite{Ra1}. We recall that we take $\Lambda(x)=0$ if $x$ is not an integer.

\begin{lemma}\label{Lem-Landau}
	Assume RH. Let $T$ be large. Suppose $a$ and $b$ are positive integers with $a>b$. Then
	\[
	\sum_{T<\gamma\le 2T}(a/b)^{i\gamma}=-\frac{T}{2\pi}\frac{\Lambda(a/b)}{\sqrt{a/b}}+O\big(\sqrt{ab}(\log T)^2\big).
	\]
\end{lemma}
\noindent If $a<b$, then we take the complex conjugate of the left-hand side above and apply the lemma to $b/a$. This yields a main term of 
\[
-\frac{T}{2\pi}\frac{\Lambda(b/a)}{\sqrt{b/a}}
\]
on the right-hand side. The next lemma is an easy consequence of Taylor's theorem.

\begin{lemma}\label{Lem-Taylor}
	Let $k>0$ and suppose $\gamma \in \cT$. Then
	\[
	\exp \bigg(2k\sum_{i=1}^{\cI}G_{i,\cI}(\gamma)\bigg)\ll \prod_{i=1}^{\cI}\left(\sum_{n=0}^{[e^2k\beta_i^{-3/4}]}\frac{(kG_{i,\cI}(\gamma))^n}{n!}\right)^2
	\]
	as $T\to \infty$. If, instead, $\gamma\in S(j)$ for some $1\le j\le \cI-1$, then
	\[
	\exp \bigg(2k\sum_{i=1}^jG_{i,j}(\gamma)\bigg)\ll \prod_{i=1}^j\left(\sum_{n=0}^{[e^2k\beta_i^{-3/4}]}\frac{(kG_{i,j}(\gamma))^n}{n!}\right)^2.
	\]
	The implied constants are independent of $k$.
\end{lemma}

\begin{proof}
	We prove the first statement, as the second follows from a similar proof. Recall from \eqref{defT} that $\gamma\in \cT$ means $|G_{i,j}(\gamma)|\le \beta_i^{-3/4}$ for all $i\le \cI$. First suppose $e^2k\beta_{i_0}^{-3/4}<1$ for some $i_0\le \cI$. Then $[e^2k\beta_i^{-3/4}]=0$ for all $i\ge i_0$. In this case, we use the trivial estimate
	\[
	\exp\left(2k\sum_{i=i_0}^{\cI}G_{i,\cI}(\gamma)\right)\le \exp\left(2k\sum_{i=i_0}^{\cI}\beta_i^{-3/4}\right).
	\]
	The sum on the right-hand side of this inequality is at most
	\[
	\frac{\beta_{i_0}^{-3/4}}{1-20^{-3/4}}\le \frac{1}{2k},
	\]
	as we have assumed $e^2k\beta_{i_0}^{-3/4}<1$. Hence
	\be\label{smallk}
	\exp\left(2k\sum_{i=i_0}^{\cI}G_{i,\cI}(\gamma)\right)\le e\prod_{i=i_0}^{\cI}\left(\sum_{n=0}^{[e^2k\beta_i^{-3/4}]}\frac{(kG_{i,\cI})^n}{n!}\right)^2,
	\ee
	since the sums on the right-hand side are identically $1$. If we may take $i_0=1$, then we are done. Thus it suffices to assume $e^2k\beta_i^{-3/4}\ge 1$ for $i<i_0$. By Taylor's theorem with explicit remainder, we have
	\be\label{taylor1}
	e^x\left(1-\frac{e^{|x|}|x|^{N+1}}{(N+1)!}\right)\le \sum_{n=0}^N\frac{x^n}{n!}
	\ee
	for $x\in \R$ and any natural number $N$. 
	We take $x=kG_{i,\cI}(\gamma)$ and $N=[e^2k\beta_i^{-3/4}]$. Using the inequality $n!\geq (n/e)^n$, it can be shown that
	\[
	\frac{e^{|x|}|x|^{N+1}}{([e^2k\beta_i^{-3/4}]+1)!}\le e^{-k\beta_i^{-3/4}}
	\]
	for any $i\le \cI$. Using this in \eqref{taylor1}, we find that
	\[
	e^{kG_{i,\cI}(\gamma)}\Big(1-e^{-k\beta_i^{-3/4}}\Big)\le \sum_{n=0}^{[e^2k\beta_i^{-3/4}]}\frac{(kG_{i,\cI}(\gamma))^n}{n!}.
	\]
	After squaring both sides of this inequality for all $i<i_0$, it follows that
	\be\label{taylor2}
	\exp \left(2k\sum_{i=1}^{i_0-1}G_{i,\cI}(\gamma)\right)\prod_{i=1}^{i_0-1}\Big(1-e^{-k\beta_i^{-3/4}}\Big)^2\le \prod_{i=1}^{i_0-1}\left(\sum_{n=0}^{[e^2k\beta_i^{-3/4}]}\frac{(kG_{i,\cI}(\gamma))^n}{n!}\right)^2.
	\ee
	The product on the left-hand side of \eqref{taylor2} is
	\be\label{taylor3}
	\ge \exp\left(-2\sum_{i=1}^{i_0-1}\beta_i^{3/4}/k\right);
	\ee
	this follows from the inequality $1-e^{-u}\ge e^{-1/u}$ for $u>0$.
	Since $e^2k\beta_{i_0-1}^{-3/4}\ge 1$, the sum here is
	\[
	\sum_{i=1}^{i_0-1}\beta_i^{3/4}\le \frac{\beta_{i_0}^{3/4}}{20^{3/4}-1}\le 2\beta_{i_0-1}^{3/4}\le 2e^2k.
	\]
	This bound with \eqref{taylor2} and \eqref{taylor3} implies
	\be\label{taylor4}
	\exp\left(2k\sum_{i=1}^{i_0-1}G_{i,\cI}(\gamma)\right)\le e^{4e^2}\prod_{i=1}^{i_0-1}\left(\sum_{n=0}^{[e^2k\beta_i^{-3/4}]}\frac{(kG_{i,j}(\gamma))^n}{n!}\right)^2.
	\ee
	Combining this inequality with \eqref{smallk}, we conclude
	\[
	\exp\left(2k\sum_{i=1}^{\cI}G_{i,\cI}(\gamma)\right)\ll \prod_{i=1}^{\cI}\left(\sum_{n=0}^{[e^2k\beta_i^{-3/4}]}\frac{(kG_{i,j}(\gamma))^n}{n!}\right)^2
	\]
	with implied constant $e^{4e^2+1}$. Lastly, suppose there is no such $i_0$, i.e. $e^2k\beta_{\cI}^{-3/4}\ge 1$. Then the argument used to derive \eqref{taylor4} may be applied to all $i\le \cI$.
\end{proof}

The following lemma gives an upper bound on mixed discrete moments of the $G_{i,j}(\gamma)$ in terms of corresponding mixed moments of the random models $G_{i,j}(X)$.

\begin{lemma}\label{Lem-mixed moments}
	Assume RH. Let $k>0$ and let $j$ be a natural number with $j\le \cI$. Let $\hat{\ell}=(\ell_1,\ldots,\ell_j)$ be a $j$-tuple in $\Z_{\geq 0}^j$ whose components satisfy $\ell_i\le 2e^2k\beta_i^{-3/4}$ for $1\le i\le j$. Then
	\[
	\sum_{0<\gamma\le T}\prod_{i=1}^jG_{i,j}^{\ell_i}(\gamma)\le N(T,2T)\,\E \left[\prod_{i=1}^jG_{i,j}(X)^{\ell_i}\right]+O\big(T^{e/25}(\log T)^2\big).
	\]
\end{lemma}

\begin{proof}
	We begin with the identity $\re(z)=\frac{1}{2}(z+\overline{z})$ and write
	\[
	G_{i,j}(\gamma)=\tfrac{1}{2}\sum_{n\in I_i}\frac{w_j(n)}{\sqrt{n}}(n^{-i\gamma}+n^{i\gamma}).
	\]
	Thus we can expand the $\ell_i$-th power of $G_{i,j}(\gamma)$ as
	\[
	2^{-\ell_i}\sum_{T<\gamma\le 2T}\sum_{n_{i,1},\ldots,n_{i,\ell_i}\in I_i}\frac{w_j(n_{i,1})\cdots w_j(n_{i,\ell_i})}	{\sqrt{n_{i,1}\cdots n_{i,\ell_i}}}\prod_{l=1}^{\ell_i}\Big(n_{i,l}^{-i\gamma}+n_{i,l}^{i\gamma}\Big),
	\]
	where $n_{i,l}$ denotes the $l$-th entry of the $\ell_i$-tuple $(n_{i,1},\ldots,n_{i,\ell_i})\in \N^{\ell_i}$. Multiplying all such expressions for $i\le j$ together and summing over $\gamma$, we see that
	\begin{multline}\label{maingammasum}
	\sum_{T<\gamma\le 2T}\prod_{i=1}^jG_{i,j}^{\ell_i}(\gamma)\\
	=2^{-(\ell_1+\cdots+\ell_j)}\sum_{T<\gamma\le 2T}\sum_{\hat{n}_1\in (I_1\cap \N)^{\ell_1}}\cdots \sum_{\hat{n}_j\in (I_j\cap \N)^{\ell_j}}\prod_{i=1}^j\prod_{l=1}^{\ell_i}\frac{w_j(n_{i,l})}{\sqrt{n_{i,l}}}\Big(n_{i,l}^{-i\gamma}+n_{i,l}^{i\gamma}\Big).
	\end{multline}
	Moving the sum over $\gamma$ inside, we ultimately need to consider sums of the form
	\be\label{gamma}
	\sum_{T<\gamma\le 2T}\prod_{i=1}^j\prod_{l=1}^{\ell_i}\Big(n_{i,l}^{-i\gamma}+n_{i,l}^{i\gamma}\Big).
	\ee
	Let $\hat{e}$ denote a $(\ell_1+\cdots +\ell_j)$-tuple in $\{-1,1\}^{\ell_1}\times \cdots \times \{-1,1\}^{\ell_j}=\{-1,1\}^{L_j}$,
	where $L_j=\ell_1+\cdots +\ell_j$. Then we may expand the double product in \eqref{gamma} as
	\be\label{gammasum}
	\sum_{\hat{e}\in \{-1,1\}^{L_j}}\prod_{i=1}^j\prod_{l=1}^{\ell_i}n_{i,l}^{ie_{i,l}\gamma},
	\ee
	where $e_{i,l}$ is the $l$-th entry ($1\le l\le \ell_i$) of the $i$-th piece $\{-1,1\}^{\ell_i}$ of $\hat{e}\in \{-1,1\}^{L_j}$. Alternatively, $e_{i,l}$ is the $(\ell_1+\cdots \ell_{i-1}+l)$-th entry of the full $L_j$-tuple. If
	\[
	\prod_{i=1}^j\prod_{l=1}^{\ell_i}n_{i,l}^{e_{i,l}}=1,
	\]
	then summing \eqref{gammasum} over $\gamma$ simply yields $N(T,2T)$, the number of $\gamma$ in $(T,2T]$. For all other terms we may apply Lemma~\ref{Lem-Landau}. Hence \eqref{gamma} is
	\begin{multline}\label{postlandau}
	N(T,2T)\,\Big(\sum_{\substack{\hat{e}\in \{-1,1\}^{L_j}\\  n_{1,1}^{e_{1,1}}\cdots n_{j,\ell_j}^{e_{j,\ell_j}}=1}}1\Big)-\frac{T}{\pi}\sum_{\substack{\hat{e}\in \{-1,1\}^{L_j}\\ n_{1,1}^{e_{1,1}}\cdots n_{j,\ell_j}^{e_{j,\ell_j}}>1}}\frac{\Lambda(n_{1,1}^{e_{1,1}}\cdots n_{j,\ell_j}^{e_{j,\ell_j}})}{\sqrt{n_{1,1}^{e_{1,1}}\cdots n_{j,\ell_j}^{e_{j,\ell_j}}}}\\
	+O\Big(2^{L_j}\sqrt{n_{1,1}\cdots n_{j,\ell_j}}(\log T)^2\Big).
	\end{multline}
	Here we have grouped together conjugate pairs in the second sum. The term involving the second sum in \eqref{postlandau} is non-positive due to the factor $-T/\pi$. Hence we may omit the whole term to obtain an upper bound. Taking this upper bound and using it in \eqref{maingammasum}, we obtain
	\begin{multline}\label{postlandausum}
	\sum_{T<\gamma\le 2T}\prod_{i=1}^jG_{i,j}^{\ell_i}(\gamma)\\
	\le N(T,2T)\sum_{\substack{\hat{n}_1\in (I_1\cap\N)^{\ell_1}\\ n_{1,l}\in I_1\\ \text{for }1\le l\le \ell_1}}\cdots \sum_{\substack{\hat{n}_j\in (I_j\cap\N)^{\ell_j}\\ n_{j,l}\in I_j\\ \text{for }1\le l\le \ell_j}}\bigg(\prod_{i=1}^j\prod_{l=1}^{\ell_i}\frac{w_j(n_{i,l})}{\sqrt{n_{i,l}}}\bigg)\Big(2^{-L_j}\sum_{\substack{\hat{e}\in \{-1,1\}^{L_j}\\ n_{1,1}^{e_{1,1}}\cdots n_{j,\ell_j}^{e_{j,\ell_j}}=1}}1\Big)\\
	+O\Big((\log T)^2 \sum_{\substack{\hat{n}_1\in (I_1\cap\N)^{\ell_1}\\ n_{1,l}\in I_1\\ \text{for }1\le l\le \ell_1}}\cdots \sum_{\substack{\hat{n}_j\in (I_j\cap\N)^{\ell_j}\\ n_{j,l}\in I_j\\ \text{for }1\le l\le \ell_j}}1\Big).	
	\end{multline}
	The error term here is
	\[
	\ll (\log T)^2\prod_{i=1}^jT^{\beta_i\ell_i}\ll T^{10e^2k e^{-250k}}(\log T)^2\le T^{e/25}(\log T)^2.
	\]
	To handle the main term, we detect the condition $n_{1,1}^{e_{1,1}}\cdots n_{j,\ell_j}^{e_{j,\ell_j}}=1$ with the expectation
	\[
	\E\Big[X_{n_{1,1}}^{e_{1,1}}\cdots X_{n_{j,\ell_j}}^{e_{j,\ell_j}}\Big]=
	\begin{cases}
	1 &\hbox{if }n_{1,1}^{e_{1,1}}\cdots n_{j,\ell_j}^{e_{j,\ell_j}}=1,\\
	0 &\hbox{otherwise}.
	\end{cases}
	\]	
	Then the leading term in \eqref{postlandausum} may be written
	\[
	N(T,2T)\sum_{\substack{\hat{n}_1\in (I_1\cap\N)^{\ell_1}\\ n_{1,l}\in I_1\\ \text{for }1\le l\le \ell_1}}\cdots \sum_{\substack{\hat{n}_j\in (I_j\cap\N)^{\ell_j}\\ n_{j,l}\in I_j\\ \text{for }1\le l\le \ell_j}}\bigg(\prod_{i=1}^j\prod_{l=1}^{\ell_i}\frac{w_j(n_{i,l})}{\sqrt{n_{l_i}}}\bigg)\Big(2^{-L_j}\sum_{\hat{e}\in \{-1,1\}^{L_j}}\E\Big[X_{n_{1,1}}^{e_{1,1}}\cdots X_{n_{j,\ell_j}}^{e_{j,\ell_j}}\Big]\Big).
	\]
	The innermost sum here is
	\[
	\sum_{\hat{e}\in\{-1,1\}^{L_j}}\E\left[\prod_{i=1}^j\prod_{l=1}^{\ell_i}X_{n_{i,l}}^{e_{i,l}}\right]=\E\left[\prod_{i=1}^j\prod_{l=1}^{\ell_i}\big(X_{n_{i,l}}^{-1}+X_{n_{i,l}}\big)\right].
	\]
	Moving the expectation outside, we see that our leading term in \eqref{postlandausum} is
	\[
	\E\bigg[2^{-L_j}\sum_{T<\gamma\le 2T}\sum_{\hat{n}_1\in (I_1\cap \N)^{\ell_1}}\cdots \sum_{\hat{n}_j\in (I_j\cap \N)^{\ell_j}}\prod_{i=1}^j\prod_{l=1}^{\ell_i}\frac{w_j(n_{i,l})}{\sqrt{n_{i,l}}}\big(X_{n_{i,l}}^{-1}+X_{n_{i,l}}\big)\bigg].
	\]
	Now we reverse our steps leading up to \eqref{maingammasum} with $n_{i,l}^{-i\gamma}$ replaced with $X_{n_{i,l}}$. This completes the proof.
\end{proof}

We are now prepared to prove an upper bound for the average of $\exp (2k\sum_{i=1}^{\cI}G_{i,\cI}(\gamma))$ over $\gamma\in \cT$. By Proposition~\ref{Prop 1}, this is approximately the corresponding average of $|\zeta'(\rho)|^{2k}$ over $\cT$.

\begin{lemma}\label{Lem-T}
	Assume RH. Let $k>0$. Then
	\[
	\sum_{\gamma\in \cT}\exp\left(2k\sum_{i=1}^{\cI}G_{i,\cI}(\gamma)\right)
	\ll N(T,2T)\,\E \left[\exp\left(2k\sum_{i=1}^{\cI}G_{i,\cI}(X_p)\right)\right]+e^{2k}T^{e/5}(\log T)^2
	\]
	as $T\to \infty$.
\end{lemma}

\begin{proof}
	By Lemma~\ref{Lem-Taylor} we have
	\be\label{presquare}
	\sum_{\gamma\in \cT}\exp\left(2k\sum_{i=1}^{\cI}G_{i,\cI}(\gamma)\right)\ll\sum_{\gamma\in \cT}\prod_{i=1}^{\cI}\left(\sum_{n=0}^{[e^2k\beta_i^{-3/4}]}\frac{(kG_{i,\cI}(\gamma))^n}{n!}\right)^2.
	\ee
	All of the terms here are squared and, hence, nonnegative. Consequently, we may extend the sum to all $T<\gamma\le 2T$ and still have an upper bound. Hence, after expanding the square, we see that the right-hand side of \eqref{presquare} is bounded from above by
	\[
	\sum_{T<\gamma\le 2T}\prod_{i=1}^{\cI}\left(\sum_{m,n=0}^{[e^2k\beta_i^{-3/4}]}\frac{k^{m+n}G_{i,\cI}(\gamma)^{m+n}}{(m!)(n!)}\right).
	\]
	We expand the product and move the sum over $\gamma$ inside to get
	\be\label{postsquare}
	\sum_{m_1,n_1=0}^{[e^2k\beta_1^{-3/4}]}\cdots \sum_{m_{\cI},n_{\cI}=0}^{[e^2k\beta_{\cI}^{-3/4}]}\frac{k^{m_1+n_1+\cdots m_{\cI}+n_{\cI}}}{(m_1!)(n_1!)\cdots (m_{\cI}!)(n_{\cI}!)}\sum_{T<\gamma\le 2T}\prod_{i=1}^{\cI}G_{i,\cI}(\gamma)^{m_i+n_i}.
	\ee
	By Lemma~\ref{Lem-mixed moments}, the inner-most sum here is
	\[
	\sum_{T<\gamma\le 2T}\prod_{i=1}^{\cI}G_{i,\cI}(\gamma)^{m_i+n_i}\le N(T,2T)\,\E \left[ \prod_{i=1}^{\cI}G_{i,\cI}(X)^{m_i+n_i}\right]+O(T^{e/25}(\log T)^2).
	\]
	Therefore \eqref{postsquare} is
	\begin{multline}\label{postexp}
	\le N(T,2T)\,\E \left[\sum_{m_1,n_1=0}^{[e^2k\beta_1^{-3/4}]}\cdots \sum_{m_{\cI},n_{\cI}=0}^{[e^2k\beta_{\cI}^{-3/4}]}\prod_{i=1}^{\cI}\frac{k^{m_i+n_i}}{(m_i!)(n_i!)}G_{i,\cI}(X)^{m_i+n_i}\right]\\
	+O\left(T^{e/25}(\log T)^2 \sum_{m_1,n_1=0}^{[e^2k\beta_1^{-3/4}]}\cdots \sum_{m_{\cI},n_{\cI}=0}^{[e^2k\beta_{\cI}^{-3/4}]}\prod_{i=1}^{\cI}\frac{k^{m_i+n_i}}{(m_i!)(n_i!)}\right).
	\end{multline}
	The $O$-term may be refactored as
	\be\label{postexpoterm}
	T^{e/25}(\log T)^2\prod_{i=1}^{\cI}\left(\sum_{n=0}^{[e^2\beta_i^{-3/4}]}\frac{k^n}{n!}\right)^2\le e^{2k}T^{e/25}(\log T)^2.
	\ee
	For the main term in \eqref{postexp}, note that $\prod_{i=1}^{\cI}\E[G_{i,\cI}(X)^{m_i+n_i}]$ is nonnegative. To see this, recall from \eqref{postlandausum} that it may be expressed as a sum of nonnegative terms. Therefore we may extend the sums to all $m_1,n_1,\ldots,m_{\cI},n_{\cI}\geq 0$ to get an upper bound.
	Hence our main term in \eqref{postexp} is
	\[
	\le N(T,2T)\,\E\left[\sum_{m_1,n_1=0}^{\infty}\cdots \sum_{m_{\cI},n_{\cI}=0}^{\infty}\prod_{i=1}^{\cI}	\frac{k^{m_i+n_i}}{(m_i!)(n_i!)}G_{i,\cI}(X)^{m_i+n_i}\right].
	\]
	This may be refactored as
	\[
	N(T,2T)\,\E\left[\prod_{i=1}^{\cI}\Big(\sum_{n=0}^{\infty}\frac{k^n}{n!}G_{i,\cI}(X)^n\Big)^2\right]=N(T,2T)\,\E \left[\exp \left(2k\sum_{i=1}^{\cI}G_{i,\cI}(X)\right)\right].
	\]
	Combining this with \eqref{postexpoterm} and \eqref{postexp}, we obtain the claimed upper bound.
\end{proof}

For the average over $S(j)$, we have to be more careful than we were for $\cT$ in Lemma~\ref{Lem-T}. This is because there are $\cI \asymp \log\log\log T$ subsets $S(j)$. We will exploit the fact that $\gamma \in S(j)$ implies $|G_{j+1,\ell}(\gamma)|\ge \beta_{j+1}^{-3/4}$ for some $j+1\le \ell\le \cI$.
\begin{lemma}\label{Lem-S(j)}
	Assume RH. Let $k>0$. For $1\le j\le \cI-1$, we have
	\begin{multline*}
	\sum_{\gamma\in S(j)}\exp \left(2k\sum_{i=1}^jG_{i,j}(\gamma)\right)
	\ll e^{-1/21\beta_{j+1}\log (1/\beta_{j+1})}N(T,2T)\E\left[\exp\left(2k\re\sum_{i=1}^jG_{i,j}(X)\right)\right]\\
	+T^{(e+5)/25}(\log T)^2
	\end{multline*}
	as $T\to \infty$. We also have
	\[
	\#S(0)\ll N(T,2T)e^{-(\log\log T)^2/10}+e^{2k}T^{(e+5)/25}(\log T)^2.
	\]
\end{lemma}

\begin{proof}
	As in the previous proof, we apply Lemma~\ref{Lem-Taylor} to see that
	\be\label{jpresquare}
	\sum_{\gamma\in S(j)}\exp\left(2k\sum_{i=1}^jG_{i,j}(\gamma)\right)\ll \sum_{\gamma\in S(j)}\prod_{i=1}^j\left(\sum_{n=1}^{[e^2k\beta_i^{-3/4}]}\frac{(kG_{i,j}(\gamma))^n}{n!}\right)^2.
	\ee
	This is valid for $0\le j< \cI$ if we take the empty sum to be $0$. By \eqref{defS(j)} and \eqref{defS(0)}, $\gamma\in S(j)$ implies $1\le \beta_{j+1}^{3/4}|G_{j+1,\ell}(\gamma)|$ for some $j+1\le \ell\le \cI$. Hence \eqref{jpresquare} is
	\be\label{jexploit}
	\le \sum_{\ell=j+1}^{\cI} \sum_{\gamma \in S(j)}\prod_{i=1}^j\left(\sum_{n=1}^{[e^2k\beta_i^{-3/4}]}\frac{(kG_{i,j}(\gamma))^n}{n!}\right)^2\Big(\beta_{j+1}^{3/4}G_{j+1,\ell}(\gamma)\Big)^{2[1/10\beta_{j+1}]}.
	\ee
	Because of the squared terms, we may extend the sum over $\gamma\in S(j)$ to $T<\gamma\le 2T$. Expand the squares and product as we did in the proof of Lemma~\ref{Lem-T}. Thus \eqref{jexploit} is bounded from above by
	\begin{multline*}
	\sum_{\ell=j+1}^{\cI}(\beta_{j+1}^{3/4})^{2[1/10\beta_{j+1}]}\sum_{m_1,n_1=0}^{[e^2k\beta_1^{-3/4}]}\cdots \sum_{m_j,n_j=0}^{[e^2k\beta_j^{-3/4}]}\frac{k^{m_1+n_1+\cdots m_j+n_j}}{(m_1!)(n_1!)\cdots (m_j!)(n_j!)}\\
	\times\sum_{T<\gamma\le 2T}G_{j+1,\ell}(\gamma)^{2[1/10\beta_{j+1}]}\prod_{i=1}^jG_{i,j}(\gamma)^{m_i+n_i}.
	\end{multline*}
	We can get an upper bound for this new expression by carefully following the proof of Lemma~\ref{Lem-T} for each $\ell$. Namely, it is
	\begin{multline}\label{jpostexp}
	\ll (\beta_{j+1}^{3/4})^{2[1/10\beta_{j+1}]}\sum_{\ell=j+1}^{\cI}N(T,2T)\E \left[G_{j+1,\ell}(X)^{2[1/10\beta_{j+1}]}\exp\left(2k\sum_{i=1}^jG_{i,j}(X)\right)\right]\\
	+T^{e/25+1/5}(\log T)^2.
	\end{multline}
	There is no $e^{2k}$ in the $O$-term in this case, since
	\[
	\big(\beta_{j+1}^{3/4}\big)^{2[1/10\beta_{j+1}]}\le e^{-750k/20}\le e^{-2k}.
	\]
	Consider the expectation in \eqref{jpostexp}. If $T$ is large, then we certainly have $T^{\beta_1}>\log T$. Likewise we also have $p^2\le T^{\beta_1}$ if $p<\log T$. By the definition of $\Lambda_{\cL}$, it follows that $G_{j+1,\ell}(X)$ and $G_{i,j}(X)$ are independent for $i\le j$ if $j\ge 1$. Thus the expectation in \eqref{jpostexp} is
	\be\label{independent}
	\E\Big[G_{j+1,\ell}(X)^{2[1/10\beta_{j+1}]}\Big]\cdot \E\left[\exp\left(2k\sum_{i=1}^jG_{i,j}(X)\right)\right].
	\ee
	This still holds for $j=0$, as the second expectation here is precisely $1$. We estimate the first expectation in \eqref{independent} for $j\ge 0$ as follows. 
	For $j\ge 1$, this first expectation is
	\[
	\frac{(2[1/10\beta_{j+1}])!}{2^{2[1/10\beta_{j+1}]}([1/10\beta_{j+1}])!}\bigg(\sum_{p\in I_{j+1}}\frac{w_{\ell}(p)^2}{p}\bigg)^{[1/10\beta_{j+1}]}\ll \bigg(\frac{1}{10e\beta_{j+1}}\sum_{T^{\beta_j}<p\le T^{\beta_{j+1}}}\frac{1}{p}\bigg)^{[1/10\beta_{j+1}]}
	\]
	by Stirling's approximation.
	The sum of reciprocal primes here is $\le 5$ for large $T$.
	It follows that
	\be\label{expbound}
	\big(\beta_{j+1}^{3/4}\big)^{2[1/10\beta_{j+1}]}\E\Big[G_{j+1,\ell}(X)^{2[1/10\beta_{j+1}]}\Big]\ll e^{-\frac{1}{2}[1/10\beta_{j+1}]\log(1/\beta_{j+1})}
	\ee
	for sufficiently large $T$ and any $\ell\ge j+1\ge 2$. Now, there are $\cI-j$ terms in the sum over $\ell$ in \eqref{jpostexp}. Observe that $\beta_{\cI}\le 1$ implies
	\[
	\cI-j\le =\frac{\log (\beta_{\cI}/\beta_j)}{\log 20}\ll \log (1/\beta_{j+1}).
	\]
	This bound along with \eqref{expbound} implies that \eqref{jpostexp} is
	\[
	\ll e^{-1/21\beta_{j+1}\log (1/\beta_{j+1})}N(T,2T)\,\E\left[\exp\left(2k\sum_{i=1}^jG_{i,j}(X)\right)\right]+T^{(e+5)/25}(\log T)^2.
	\]
	for $j\ge 1$. This proves the first claim. For $j=0$, the expectation
	\[
	\E\Big[G_{j+1,\ell}(X)^{2[1/10\beta_{j+1}]}\Big]
	\]
	from \eqref{jexploit} is slightly more complicated than the $j\ge 1$ case. This is because $G_{1,\ell}(X)$ includes some nonzero terms corresponding to squared primes, and clearly $X_p$ and $X_{p^2}=X_p^2$ are not independent. 
	Since $w_{\ell}(n)\le 1$, the sum over squared primes in $G_{1,\ell}(X)$ is at most
	\[
	\sum_{p\le \log T}\frac{w_{\ell}(p^2)}{p}\le 2\log\log\log T
	\]
	for large $T$. Thus
	\[
	G_{1,\ell}(X)^{2[1/10\beta_1]}
	\ll 2^{2[1/10\beta_1]}\bigg(\re \sum_{p\in I_1}\frac{w_{\ell}(p)}{\sqrt{p}}X_p\bigg)^{2[1/10\beta_1]}+2^{4[1/10\beta_1]}(\log\log\log T)^{2[1/10\beta_1]}.
	\]
	Insert this bound in \eqref{jpostexp}. This yields the upper bound
	\[
	\ll N(T,2T)\bigg\{\bigg(\frac{4\beta_1^{1/2}}{10e}\sum_{p\in I_1}\frac{w_{\ell}(p)}{p}\bigg)^{[1/10\beta_1]}+(4\beta_1^{3/4}\log\log\log T)^{2[1/10\beta_1]}\bigg\}
	+T^{(e+5)/25}(\log T)^2.
	\]
	Since $4\log\log\log T\le (\log\log T)^{1/2}$ for large $T$, the main term here is
	\[
	\ll N(T,2T)\Big(e^{-[1/10\beta_1]}+e^{-2[1/10\beta_1]\log\log\log T}\Big)\ll N(T)e^{-(\log\log T)^2/10}
	\]
	as claimed.
\end{proof}

\section{Proof of Proposition~\ref{Prop 2}}
\noindent Observe that
\[
J_k(T)=\sum_{\gamma\in \cT}|\zeta'(\rho)|^{2k}+\sum_{j=1}^{\cI-1}\sum_{\gamma\in S(j)}|\zeta'(\rho)|^{2k}+\sum_{\gamma\in S(0)}|\zeta'(\rho)|^{2k}.
\]
It suffices to estimate each of these pieces individually.

\subsection{The sum over $\cT$}
Using the inequality \eqref{newinequality} with $j=\cI$, we have
\[
\sum_{\gamma\in \cT}|\zeta'(\rho)|^{2k}\ll_k(\log T)^{2k}\sum_{\gamma\in \cT}\exp\left(2k\sum_{i=1}^{\cI}G_{i,\cI}(\gamma)\right);
\]
we have included the factor $e^{2k/\beta_{\cI}}$ in the implied constant, since $\beta_{\cI}\approx e^{-1000k}$.
By Lemma~\ref{Lem-T}, the right-hand side is
\be\label{cT}
\ll_k N(T,2T)\, (\log T)^{2k}\E\left[\exp 2k\sum_{i=1}^{\cI}G_{i,\cI}(X)\right]+T^{(e+5)/25}(\log T)^{2k+2}.
\ee
For large $T$, no two intervals $I_1,\ldots,I_{\cI}$ contain powers of the same prime. That is, we may use independence of the random variables $X_p$ to write the expectation in \eqref{cT} as
\be\label{cTexp}
\prod_{i=1}^{\cI}\E\big[\exp 2kG_{i,\cI}(X)\big].
\ee
For $i\ge 2$, we recall that
\[
G_{i,\cI}(X)=\sum_{p\in I_i}\frac{w_{\cI}(p)}{\sqrt{p}}X_p.
\]
By a standard calculation, we have
\[
\E\big[\exp\big(2kG_{i,\cI}(X)\big)\big]=\prod_{p\in I_i}I_0\bigg(2k\frac{w_{\cI}(p)}{\sqrt{p}}\bigg),
\]
where $I_0(z)=\sum_{n=0}^{\infty}\frac{(z/2)^{2n}}{(n!)^2}$ is the modified Bessel function of the first kind.

Now consider the $i=1$ term in \eqref{cTexp}. For any prime $p>\log T$, the calculation is exactly as it was above for $i\ge 2$. For $p\le \log T$, we have both $X_p$ and $X_p^2$ appearing in the expression for $G_{1,\cI}(X)$. For these, we must consider
\[
\E\left[\exp \left(2k\re\frac{w_{\cI}(p)}{\sqrt{p}}X_p+2k\re \frac{w_{\cI}(p^2)}{p}X_p^2\right)\right].
\]
Note that we may replace $\re (X_p^2)$ with $2(\re X_p)^2-1$; this is a consequence of the double angle formula for cosine. Hence the above may be written
\be\label{exp1}
\E \left[\exp \left(2k\frac{w_{\cI}(p)}{\sqrt{p}}\re X_p\right)\cdot\exp \left(4k\frac{w_{\cI}(p^2)}{p}(\re X_p)^2\right)\right]\exp\left(-2k\frac{w_{\cI}(p^2)}{p}\right).
\ee
The expectation here is
\begin{align*}
&= \sum_{m=0}^{\infty}\frac{(2k)^m}{m!}\left(\frac{w_{\cI}(p)}{\sqrt{p}}\right)^m\sum_{n=0}^{\infty}\frac{(4k)^n}{n!}\left(\frac{w_{\cI}(p^2)}{p}\right)^n\E[(\re X_p)^{m+2n}]\\
&= \sum_{h=0}^{\infty}\frac{k^{2h}}{(2h)!}\left(\frac{w_{\cI}(p)}{\sqrt{p}}\right)^{2h}\sum_{n=0}^{\infty}\frac{k^n}{n!}\left(\frac{w_{\cI}(p^2)}{p}\right)^n\binom{2(h+n)}{h+n}.
\end{align*}
This follows from direct calculation of the moments
\[
\E[(\re X_p)^{m+2n}]=
\begin{cases}
\binom{2(h+n)}{h+n}2^{-2(h+n)}&\hbox{if } m=2h,\\
0 &\hbox{if } m \text{ is odd}.
\end{cases}
\]
Note that $w_{\cI}(p)\le 1$ and $w_{\cI}(p^2)\le \frac12$. Isolating the first three terms (those for which $h,n\le 1$) of the double sum and using the elementary inequality $\binom{2(h+n)}{h+n}\le 2^{2(h+n)}$ for the rest of the terms, we see that the expectation in \eqref{exp1} is
\[
1+\frac{k^2w_{\cI}(p)}{p}+\frac{2kw_{\cI}(p^2)}{p}+O\bigg(\frac{e^{3k}}{p^2}\bigg)=I_0\bigg(2k\frac{w_{\cI}(p)}{\sqrt{p}}\bigg)\exp\bigg(2k\frac{w_{\cI}(p^2)}{p}\bigg)\big(1+O_k(1/p^2)\big).
\]
It follows that \eqref{exp1} is
\[
\le I_0\Big(2k\frac{w_{\cI}(p)}{\sqrt{p}}\Big)\big(1+O_k(1/p^2)\big),
\]
and so
\[
\E\bigg[\exp\bigg(2k\re\sum_{n\le T^{\beta_{\cI}}}\frac{w_{\cI}(n)}{\sqrt{n}}X_n\bigg)\bigg]\le \prod_{p\le T^{\beta_j}}I_0\left(2k\frac{w_{\cI}(p)}{\sqrt{p}}\right)\cdot \prod_{\tilde{p}\le \log T}\big(1+O_k(1/\tilde{p}^2)\big).
\]
Note that, as $T\to \infty$, the product over $\tilde{p}$ converges to some constant depending only on $k$. Since $I_0(2x)\le e^{x^2}$, $w_{\cI}(p)\le 1$ and $w_{\cI}(p^2)\le \frac{1}{2}$, we conclude that this expectation is
\[
\ll_k\exp\bigg(k^2\sum_{p\le T^{\beta_{\cI}}}\frac{1}{p}\bigg)\ll_k (\log T)^{k^2}.
\]
This last bound along with \eqref{cTexp} implies
\[
\sum_{\gamma\in\cT}|\zeta'(\rho)|^{2k}\ll_k(\log T)^{k(k+2)}+T^{(e+5)/25}(\log T)^{2k+2}.
\]

\subsection{The sums over $S(j)$}
Consider $1\le j\le \cI-1$. We proceed as we did for the sum over $\cT$, but we use Lemma~\ref{Lem-S(j)} instead of Lemma~\ref{Lem-T}. By \eqref{newinequality} we have
\be\label{sj1}
\sum_{\gamma\in S(j)}|\zeta'(\rho)|^{2k}\ll_k e^{2k\beta_j^{-1}}(\log T)^{2k}\sum_{\gamma\in S(j)}\exp\left(2k\sum_{i=1}^jG_{i,j}(\gamma)\right).
\ee
Lemma~\ref{Lem-S(j)} implies that the right-hand side is
\begin{multline*}
\ll_k e^{2k\beta_j^{-1}}e^{-1/21\beta_{j+1}\log(1/\beta_{j+1})}N(T,2T)\, (\log T)^{2k}\E\left[\exp\left(2k\sum_{i=1}^jG_{i,j}(X)\right)\right]\\
+e^{2k\beta_j^{-1}}T^{(e+5)/25}(\log T)^{2k+2}.
\end{multline*}
The expectation here is estimated like it was in the case $j=\cI$. We also recall that $\beta_{j+1}=20\beta_j$. It follows that the right-hand side of \eqref{sj1} is
\[
\ll_k e^{2k\beta_j^{-1}}e^{-1/420\beta_j\log(1/\beta_{j+1})}N(T,2T)(\log T)^{k^2}\\
+e^{2k\beta_j^{-1}}T^{(e+5)/25}(\log T)^{2k+2}.
\]
Note that $$2k-\frac{1}{420}\log(1/\beta_{j+1})\le-\frac{8k}{21},$$
since $\beta_{j+1}\le e^{-1000k}$ for $j\le \cI-1$. Hence we see that
\be\label{sj2}
\sum_{\gamma \in S(j)}|\zeta'(\rho)|^{2k}\ll_k e^{-8k/21\beta_j}N(T,2T)(\log T)^{k^2}+e^{2k\beta_j^{-1}}T^{(e+5)/25}(\log T)^{2k+2}.
\ee
Observe that
\[
\sum_{j=1}^{\cI-1}e^{-8k/21\beta_j}\ll_k 1.
\]
Thus summing \eqref{sj2} over $1\le j\le \cI-1$ yields
\[
\sum_{j=1}^{\cI-1}\sum_{\gamma\in S(j)}|\zeta'(\rho)|^{2k}
\ll_k N(T,2T)(\log T)^{k^2}+T^{(e+5)/25}(\log T)^{2k+2}.
\]
\subsection{The sum over $S(0)$}
By H\"older's inequality, we have
\[
\sum_{\gamma\in S(0)}|\zeta'(\rho)|^{2k}\le \Big(\sum_{\gamma\in S(0)}1\Big)^{1/q}\Big(\sum_{T<\gamma\le 2T}|\zeta'(\rho)|^{2[2k+1]}\Big)^{1/p},
\]
where $p=[2k+1]/k$, $q=1-\frac1p$ and $[x]$ is the greatest integer less than or equal to $x$. 
We estimate the first sum on the right-hand side with the second part of Lemma~\ref{Lem-S(j)}. For the second sum we use \eqref{Mi} with
$\varepsilon=1$. Note that $p\ge 2$ and $q\le \frac12$. It follows that
\begin{align*}
\sum_{\gamma\in S(0)}|\zeta'(\rho)|^{2k}&\ll_k N(T,2T)\,e^{-(\log\log T)^2/10q}(\log T)^{([2k+1]^2+1)/p}\\
&\le N(T,2T)\,\exp\big(-(\log\log T)^2/5+2([k+2]^2+1)\log\log T\big),
\end{align*}
which is $\ll N(T,2T)$ as $T\to \infty$. Combining this with the estimates for the sums over $\cT$ and the other $S(j)$ completes the proof.

\section{Sketch of the proof of Theorem~\ref{Thm 2}}
\noindent Here we describe how to modify the proof of Theorem~\ref{Thm 1} in order to prove Theorem~\ref{Thm 2}. Similar to Proposition~\ref{Prop 2}, we consider
\[
\sum_{T<\gamma\le 2T}|\zeta(\rho+\alpha)|^{2k}
\]
for a complex number $\alpha$ with $|\alpha|\le (\log T)^{-1}$. By the functional equation for the zeta-function, it suffices to assume $\re(\alpha)\ge 0$ (see the proof of Theorem 1.2 in \cite{Mi}). The approach is largely the same as it was for moments of $|\zeta'(\rho)|$. We modify the weight $w_j(n)$ introduced in Section 4 by defining
\[
w_j(n;\alpha)=\frac{w_j(n)}{n^{\re(\alpha)}}.
\]
This leads us to define
\[
G_{i,j}(t;\alpha)=\re\sum_{n\in I_i}\frac{w_j(n;\alpha)}{\sqrt{n}}n^{-i(t+\im(\alpha))}
\]
and similarly
\[
G_{i,j}(X;\alpha)=\re\sum_{n\in I_i}\frac{w_j(n;\alpha)}{\sqrt{n}}X_n.
\]
We use the inequality
\[
\log |\zeta(\rho+\alpha)|\le \sum_{i=1}^jG_{i,j}(\gamma;\alpha)+\frac{\log T}{\log x}+O(1),
\]
which is essentially Harper's~\cite{Ha} Proposition~\ref{Prop 1}. The proof of the upper bound
\[
\sum_{T<\gamma\le 2T}|\zeta(\rho+\alpha)|^{2k}\ll_k(\log T)^{k^2}
\]
claimed at the end of \S 2 relies on the obvious analogue of Lemma~\ref{Lem-mixed moments}, where $G_{i,j}(\gamma;\alpha)$ takes the place of $G_{i,j}(\gamma)$. The key difference, in comparison with \eqref{gamma} and \eqref{gammasum}, is that we must consider sums of the form
\be\label{alphasum}
\sum_{T<\gamma\le 2T}\prod_{i=1}^j\prod_{l=1}^{\ell_i}n_{i,l}^{ie_{i,l}(\gamma+\im(\alpha))}.
\ee
The diagonal terms are still those for which
\[
\prod_{i=1}^j\prod_{l=1}^{\ell_i}n_{i,l}^{e_{i,l}}=1,
\]
so we obtain the same main term as in the proof of Lemma \ref{Lem-mixed moments}. Again, we use Lemma~\ref{Lem-Landau} to handle the off-diagonal terms, i.e., those with
\[
\prod_{i=1}^j\prod_{l=1}^{\ell_i}n_{i,l}^{e_{i,l}}\neq1;
\]
as in the proof of Lemma~\ref{Lem-mixed moments}, we can assume this double product is $>1$ by combining terms which are complex conjugates. Separating the $n_{i,l}^{ie_{i,l}\im (\alpha)}$ factors from the $n_{i,l}^{ie_{i,l}\gamma}$, we see that \eqref{alphasum} is
\be\label{alphalandau}
=-\frac{T}{2\pi}\frac{\Lambda(n_{1,1}^{e_{1,1}}\cdots n_{j,\ell_j}^{e_{j,\ell_j}})}{\sqrt{n_{1,1}^{e_{1,1}}\cdots n_{j,\ell_j}^{e_{j,\ell_j}}}}\prod_{i=1}^j\prod_{l=1}^{\ell_i}n_{i,l}^{ie_{i,l}\im(\alpha)}\\
+O\Big(\sqrt{n_{1,1}^{e_{1,1}}\cdots e_{j,\ell_j}^{e_{j,\ell_j}}}(\log T)^2\Big).
\ee
Take real parts like in \eqref{postlandau}. The sign of the leading term here depends on the sign of
\be\label{cos}
\re \prod_{i=1}^j\prod_{l=1}^{\ell_i}n_{i,l}^{ie_{i,l}\im(\alpha)}=\cos \bigg(\im(\alpha) \log \prod_{i=1}^j\prod_{l=1}^{\ell_i}n_{i,l}^{e_{i,l}}\bigg).
\ee
Recall that $n_{i,l}\le T^{\beta_i}$ and $\ell_i\le 2e^2k\beta_i^{-3/4}$. This implies
\[
\prod_{i=1}^j\prod_{l=1}^{\ell_i}n_{i,l}^{e_{i,l}}\le T^{10e^2ke^{-250k}}.
\]
Note that $10e^ke^{-250k}$ is at most $e/25$. Since $|\im(\alpha)|\le (\log T)^{-1}$, we see that
\[
\bigg|\im(\alpha) \log \prod_{i=1}^j\prod_{l=1}^{\ell_i}n_{i,l}^{e_{i,l}}\bigg|< \frac{\pi}{2}.
\]
It follows that \eqref{cos} is positive, and consequently the leading term in \eqref{alphalandau} is negative. Thus we may ignore the leading term and still obtain an upper bound. The rest of the proof proceeds as before.

\begin{acknowledgements}\label{ackref}
The work herein comprised part of the author's Ph.D. thesis at the University of Rochester. The author is grateful to Prof. Steven Gonek for inspiring work on the problem and would also like to thank the referees for their helpful comments and suggestions.

The author also extends their sincere gratitude to the Leverhulme Trust (RPG-2017-320) for postdoctoral fellowship support through the research project grant ``Moments of $L$-functions in Function Fields and Random Matrix Theory". 
\end{acknowledgements}

\end{document}